\documentclass[leqno,11pt]{amsart}
\usepackage[utf8]{inputenc}
\usepackage{graphicx}
\usepackage{amscd}
\usepackage{amsmath}
\usepackage{amssymb}
\usepackage{caption}
\usepackage{amsfonts}
\usepackage{amssymb}
\usepackage{mathrsfs}
\usepackage{multicol}
\usepackage{color}
\usepackage{todonotes} 

\usepackage{esint} 

\usepackage[noadjust]{cite}

\usepackage[left=2.5cm, right=2.5cm, top=2cm, bottom=2cm]{geometry}
\usepackage{enumitem,graphicx}
\usepackage[colorlinks=true,urlcolor=blue,
citecolor=red,linkcolor=blue,linktocpage,pdfpagelabels,
bookmarksnumbered,bookmarksopen]{hyperref}

\newcommand{\wt}{\widetilde}

\newcommand{\eps}{\varepsilon}

\newcommand\supp{\mathrm{supp}\,}
\newcommand\lan{\langle}
\newcommand\ran{\rangle}

\newcommand\N{\mathbb{N}}
\newcommand\R{\mathbb{R}}

\newcommand\C{\mathbb{C}}

\newcommand{\cC}{{\mathcal C}}
\newcommand{\cD}{{\mathcal D}}
\newcommand{\cE}{{\mathcal E}}

\newcommand{\cM}{{\mathcal M}}

\newcommand{\cQ}{{\mathcal Q}}

\newcommand{\cU}{{\mathcal U}}

\newcommand{\cX}{{\mathcal X}}

\newcommand{\weakto}{\rightharpoonup}

\newcommand{\tu}{\widetilde{u}}

\newcommand{\loc}{\mathrm{loc}}

\newenvironment{altproof}[1]
{\noindent
	{\em Proof of {#1}}.}
{\nopagebreak\mbox{}\hfill $\Box$\par\addvspace{0.5cm}}

\newtheorem{theorem}{Theorem}[section] 
\newtheorem{lemma}[theorem]{Lemma}
\newtheorem{remark}[theorem]{Remark}

\newtheorem{proposition}[theorem]{Proposition}

\usepackage{mathrsfs}

\allowdisplaybreaks
\definecolor{darkgreen}{rgb}{0.09, 0.45, 0.27}
\definecolor{debianred}{rgb}{0.84, 0.04, 0.33}
\allowdisplaybreaks[1]

\numberwithin{equation}{section}

\begin{document}

\title[Traveling waves for nonlinear Schr\"odinger equations]
{Traveling waves for nonlinear Schr\"odinger equations}
	\author[L. Baldelli]{Laura Baldelli}
	\author[B. Bieganowski]{Bartosz Bieganowski}
	\author[J. Mederski]{Jaros\l aw Mederski}

	\address[L. Baldelli]
	{\newline\indent 
			IMAG, Departamento de An\`alisis Matem\`atico,
			\newline\indent 
			Universidad de Granada,
			\newline\indent 
			Campus Fuentenueva, 18071 Granada, Spain}
	\email{\href{mailto:labaldelli@ugr.es}{labaldelli@ugr.es}}

\address[B. Bieganowski]{\newline\indent
			Faculty of Mathematics, Informatics and Mechanics, \newline\indent
			University of Warsaw, \newline\indent
			ul. Banacha 2, 02-097 Warsaw, Poland}	
			\email{\href{mailto:bartoszb@mimuw.edu.pl}{bartoszb@mimuw.edu.pl}}	
 
	\address[J. Mederski]
	{\newline\indent 
			Institute of Mathematics,
			\newline\indent 
			Polish Academy of Sciences,
			\newline\indent 
			ul. \'Sniadeckich 8, 00-656 Warsaw, Poland}
	\email{\href{mailto:jmederski@impan.pl}{jmederski@impan.pl}}

\begin{abstract}
We look for traveling wave solutions to the nonlinear Schr\"odinger equation with a subsonic speed, covering  several physical models with Sobolev subcritical nonlinear effects. Our approach is based on a variant of Sobolev-type inequality involving the momentum and we show the existence of its minimizers solving the nonlinear Schr\"odinger equation.

\bigskip
 {\small
	
}

\end{abstract}
	
	\maketitle

\section*{Introduction}
\setcounter{section}{1}

In this paper, we study the existence of nontrivial finite energy  traveling waves moving with subsonic speed for a class of nonlinear Schr\"odinger equations of the type
\begin{equation}\label{pphi}
i\partial_t \Phi =\Delta\Phi+F(|\Phi|^2)\Phi\quad \hbox{in }\R^d\times\R,
\end{equation}
where $d\ge3$, $\Phi:\R^d\times \R\to\C$ satisfies the "boundary condition" $|\Phi|\to r_0$ as $|x|\to\infty$, with $r_0>0$ and $F$ is a real-valued function on $\R^+$ with $F(r_0^2)=0$.
If \eqref{pphi} is defocusing, that is if $F'(r_0^2)<0$, a simple scaling, see \cite{maris}, enables us to assume $r_0=1$ and $F'(r_0^2)=-1$, so that the sound velocity at infinity associated to \eqref{pphi} is $v_s=r_0\sqrt{-2F'(r_0^2)}=\sqrt{2}$. 
Equations of this type appear in mathematical
physics and have been derived as models of several physical phenomena such as superconductivity, superfluidity in Hellium II, Bose-Einstein condensates as well as dark optical solitons \cite{Abitetal,Kivshar,jpr}.

The Gross-Pitaevskii equation is a particular case of \eqref{pphi} with $F(s)=1-s$ and it was deeply studied in the literature, see \cite{br,maris,BethuelSautI,BethuelSautII,Gravejat,marisNonEx} and references therein. Recall that nontrivial finite energy traveling waves for $F(s)=1-s$ in dimension $d=1$ are unique up to rotation and translation and are given by
$$\Phi_c(x)=\sqrt{\frac{2-c^2}{2}}\tanh\Big(\frac{\sqrt{2-c^2}}{2}x\Big)+i\frac{c}{\sqrt{2}}$$
for $c\in (0,\sqrt{2})$. $\Phi_c$ are called {\em dark solitons}, $\Phi(x,t)=\Phi_c(x-ct)$ solves \eqref{pphi} and observe that $|\Phi(x,t)|\to 1$ as $|x|\to\infty$.

In the present paper, we are focused on {\em traveling wave} solutions which are solutions to \eqref{pphi} of the form
\begin{equation}\label{ansatz}
\Phi(x,t)=1+u(x_1-ct,y)
\end{equation}
with $u:\R^d\to\C$, $x=(x_1,y)\in\R\times\R^{d-1}$, $d\geq 2$ and $c\in\R^+$ describes the {\em speed}.
By the ansatz \eqref{ansatz} the equation for the profile $u$ is given by
\begin{equation}
\label{eq}
-ic\partial_{x_1}u+\Delta u+(1+u)F(|1+u|^2)=0\quad \hbox{in }\R^d.
\end{equation}
The study of finite energy traveling solutions for \eqref{pphi} has a long story. In the particular case of the Gross-Pitaevskii equation, this problem was first studied in the papers by Putterman and Roberts \cite{jpr}, where formal calculations and numerical analysis gave rise to a set of conjectures regarding existence, asymptotic behavior and stability of finite energy traveling waves: the so-called Jones-Putterman-Roberts program.
In particular, the existence of finite energy traveling waves is expected if and only if the propagation velocity is subsonic, namely if $c\in(0, \sqrt{2})$.
In the last years, much progress has been made to give rigorous proofs of those conjectures. Nonexistence of nontrivial finite energy traveling waves for supersonic speed was proved by Gravejat in \cite{Gravejat}, for general nonlinearities see \cite{marisNonEx}.
Concerning the asymptotics of finite energy solutions, for any $d\ge2$, finite energy solutions of \eqref{eq} converge at infinity to a fixed complex number of modulus $1$. 
We recall the existence results for small values of $c$, see \cite{BethuelSautI} in dimension 2
and \cite{bos} in dimension 3, but a complete existence result in the sub-sonic case remained for many years as a standing open problem. Finally, Mari\c{s} proved in \cite{maris} the existence
result for any subsonic speed in dimension $d \ge3$. His approach is, summing up, to minimize the Lagrangian under a Poho\v{z}aev-type constraint. Once this is accomplished, Mari\c{s} proves that
the corresponding Lagrange multiplier is $0$, concluding the proof. This approach works also for more general nonlinearities with nonvanishing conditions at infinity, such as cubic-quintic nonlinearity. As commented in \cite{maris}, this minimization approach breaks
down in dimension $2$ because of different scaling properties: the infimum is $0$ and is never attained. Moreover Chiron and Mari\c{s} \cite{ChironMaris} have considered also travelling waves with at the fixed momentum.
Another approach is due to Bellazzini and Ruiz \cite{br} and is based on the approximation domains and result of Mari\c{s} in dimension 2 and 3 were recovered. Let us also mention that Mari\c{s} and Mur \cite{marisMur} have recently considered periodic traveling waves in $\R^2$.

The aim of this work is to consider more general growth assumptions at infinity concerning $F$ in comparison to \cite{maris} in case $d\geq 4$, and provide a direct and simpler variational approach based on a Sobolev-type inequality involving the momentum, see \eqref{eq:ineq} below. Moreover, in contrary to  \cite{maris}, we no longer require regularity of minimizers on the Poho\v{z}aev manifold to solve \eqref{eq} (cf. Remark \ref{rem:final}), and we show that only the nonvanishing analysis  is sufficient to establish the existence of travelling waves instead of the full concentration-compactness.  We also point out that, under condition (F2), the bootstrap argument is no longer available.

Concerning the nonlinearity $F:\R\to\R$ we consider the following general assumptions:
\begin{itemize}
	\item[(F1)] $F\in \cC[0,\infty)$ and $\cC^1$ in a neighborhood of $1$, $F(1)=0$ and $F'(1)=-1$,
	\item[(F2)] $\limsup_{s\to\infty}|F(s)|/s^{\frac{2^*-2}{2}}=0$,
\end{itemize}
where $2^*=\frac{2d}{d-2}$ is the critical Sobolev exponent.
In particular, (F2) is in the spirit of Berestycki-Lions \cite{BerLions}.

Letting $V(s):=\int_s^1F(t)\,dt$ for $s \geq 0$, we can define the {\em energy} and the Lagrangian of the problem consider, respectively as
\begin{align*}
\cE(u)&=\frac12\int_{\mathbb R^d} |\nabla u|^2+V(|1+u|^2)\,dx,\quad u\in\cD^{1,2}(\R^d), \\
E_c(u)&=\cE(u)+c\cQ(u)=\frac12\int_{\mathbb R^d} |\nabla u|^2+V(|1+u|^2)\,dx+c\cQ(u),\quad u\in\mathcal{X},
\end{align*}
where $\cD^{1,2}(\R^d)$ stands for the completion of $\cC_0^{\infty}(\R^d)$ with respect to the norm
$\|u\|=|\nabla u|_2$,
\begin{equation}\label{X}
\cX:=\big\{u\in \cD^{1,2}(\R^d):\;  |1+u|-1\in L^2(\R^d)\big\},
\end{equation}
and $\cQ(u)$ is the first component of the {\em momentum} defined as
$$\cQ(u)=\frac12\int_{\mathbb R^d}\langle i\partial_{x_1}u,u\rangle\,dx.$$
Here and in what follows $|\cdot|_k$ denotes the $L^k$-norm, $k\in [1,\infty]$ and $ \lan f,g \ran=\Re(f)\Re(g)+\Im(f)\Im(g)$, with $\Re(\cdot), \Im(\cdot)$ the real and the imaginary part, respectively. Observe that, in view of (F1) and (F2), $\cE(u)<\infty$ if and only if $u\in\cX$.
It is clear that $\cQ$ is a well defined functional in $H^1(\R^d)$, but in general  there exist functions $u\in\cX\setminus H^1(\R^d)$ such that $\langle i\partial_{x_1}u,u\rangle\notin L^1(\R^d)$. Therefore  a proper definition of the momentum in $\cX$ was proposed in \cite[Section 2]{maris} and  then $|\cQ(u)|<\infty$ for $u\in\cX$.

Now we state the {\em Sobolev-type inequality} involving the momentum $\cQ$ for $d\geq 4$ as follows. Let $S_c\in\R$ be the largest possible constant such that  
\begin{equation}\label{eq:ineq}
\left(\int_{\R^d}|\nabla_y u|^2\, dx\right)^{\frac{d-1}{d-3}}\geq -S_c\left(\frac{1}{2}\int_{\R^d}|\partial_{x_1} u|^2+V( |1+u|^2)\,dx+c\cQ(u)\right)
\end{equation}
holds for all $u\in\cX$, where $x = (x_1, y)$, $y \in \R^{d-1}$.

Our main result can be summarized as follows.
\begin{theorem}\label{thm:1} Assume that (F1), (F2) hold and let $c\in (0,\sqrt{2})$, $d\geq 4$. There holds.\\
(a) $S_{c}>0$. \\
(b)
If $(u_n)\subset\cX$ is a sequence such that
$$
\frac{ \left(\int_{\R^d}|\nabla_y u_n|^2\, dx \right)^{\frac{d-1}{d-3}}}{\frac{1}{2}\int_{\R^d}|\partial_{x_1} u_n|^2 +V(|1+u_n|^2)\,dx+c\cQ(u_n) } \to -S_c.
$$
as $n\to\infty$,
then there is $(x_n)\subset \R^d$ such that, passing to a subsequence, 
$(u_n(\cdot+x_n))$ is weakly convergent in $\cD^{1,2}(\R^d)$ to a minimizer $u$ of \eqref{eq:ineq}, i.e. the equality in \eqref{eq:ineq} holds.\\
(c) If $u$ is a minimizer of \eqref{eq:ineq}, then (up to some rescaling)  $u$ is a weak solution to \eqref{eq}. 
\end{theorem}

Problem~\eqref{eq} was solved in~\cite[Theorem~1.1]{maris} under a slightly more restrictive growth condition at infinity, namely
$
\limsup_{s\to\infty}|F(s)|/s^{p_0}=0
$
for some \( p_0 < \frac{2^*-2}{2} \). This assumption was crucial, in particular, for carrying out the bootstrap argument.
We present a different technique to prove Theorem \ref{thm:1}, which is based on a variant of Lions' lemma  based on the nonvanishing property from \cite{maris} -- see Lemma \ref{lem:vanishing} below, and on the Sobolev-type inequality. The variational approach seems to be simpler and more direct.

If $d=3$, then we no longer have a Sobolev-type inequality, however, recall that in \cite[Theorem 6.2]{maris} a minimizer of $E_c$ on the Poho\v{z}aev manifold 
$$\cM=\left\{u\in\cX\setminus\{0\} \ : \  \frac{1}{2}\int_{\R^3}|\partial_{x_1} u_n|^2 +V(|1+u_n|^2)\,dx+c\cQ(u_n)=0\right\}$$
has been found, which solves \eqref{eq}, however it is not clear whether our approach can be applied in $\R^3$.

In what follows we state the following Lions' type lemma (cf. \cite{LionsAnn1}) which will be useful to prove the non-triviality of the solution given in the main result of the paper.

First, we introduce the Ginzburg-Landau energy as follows. Let $\varphi \in \cC^\infty (\R)$ be an odd function that $\varphi(s) = s$ for $s \in [0,2]$, $\varphi(s) = 3$ for $s \geq  4$, and $0 \leq \varphi'(s) \leq 1$ for $s \in \R$. Then, for $u \in \cX$ we set
$$
E_{GL} (u) := \frac12 \int_{\R^d} |\nabla u|^2 \, dx + \frac14 \int_{\R^d} \left( \varphi^2( |1+u| ) - 1 \right)^2 \, dx.
$$
Note that the function space $\mathcal{X}$  in \eqref{X}, as observed by \cite{maris}, for $d\ge3$ satisfies
$$\mathcal{X}=\{u\in \cD^{1,2}(\R^d) \ : \ \varphi^2(|1+u|)-1\in L^2(\R^d)\}=\{u\in  \cD^{1,2}(\R^d) \ : \ E_{GL}(u)<\infty\}.$$
Observe that for $d=3$, $\mathcal{X}$ is not a vector space.
Moreover, if $d\ge3$, then $H^1(\R^d)\subset\mathcal{X}$ and the Lagrangian $E_c(u)$ is well-defined for $u\in\mathcal{X}$.
While, for $d \in \{3,4\}$, $\mathcal{X}$ can be also written as
$$\mathcal{X}=\{u\in \cD^{1,2}(\mathbb R^d) \ : \ |1+u|^2-1\in L^2(\mathbb R^d)\},$$
see also \cite{Gerard}.

\begin{lemma}\label{lem:vanishing}
Assume $d\ge 3$ and let $(u_n)\subset \mathcal{X}$ be a sequence such that $(E_{GL}(u_n))$ is bounded and
\begin{equation}\label{eq:vanishing2}
\lim_{n\to\infty}\sup_{y\in\R^d}\int_{B(y,1)}|u_n|^2+\left(\varphi^2(|1+u_n|) - 1\right)^2\,dx=0.
\end{equation}
Then $\liminf_{n\to\infty}E_c(u_n) = \liminf_{n\to\infty} E_{GL}(u_n) + c \mathcal{Q}(u_n) \geq 0$ for any $c \in (0,\sqrt{2})$.
\end{lemma}

\begin{remark}
Recall that the classical version of the Lions' lemma \cite{LionsAnn1} states that if $(u_n)$ is bounded in $H^1(\R^d)$ and $\lim_{n\to\infty}\sup_{y\in\R^d}\int_{B(y,1)}|u_n|^2\,dx=0,$
then 
\begin{equation}\label{eq:LionsLp}
u_n\to 0\quad\hbox{in }L^p(\R^d)
\end{equation}
for $2<p<2^*$,  \cite{LionsAnn1}. Observe that, in a similar way as in Lemma \ref{lem:vanishing}, we can replace \eqref{eq:LionsLp} by the equivalent condition 
$\liminf_{n\to\infty}J(u_n)\geq 0$, where 
$$J(u):=\frac12\int_{\R^d}|\nabla u|^2+|u|^2\,dx-\frac1p\int_{\R^d}|u|^p\,dx$$ is the energy functional associated to the problem 
$-\Delta u+u=|u|^{p-2}u$. Indeed,
note that if $(u_n)$ is such a sequence as above and \eqref{eq:LionsLp} holds, then clearly $\liminf_{n\to\infty}J(u_n)\geq 0$. On the other hand, suppose that $(u_n)$ is  such a sequence as above, but $|u_n|_p$ is bounded away from $0$ passing a subsequence. Then 
 $t_n^{p-2}:=(|\nabla u_n|_2^2+|u_n|_2^2)/|u_n|_p^p$  is bounded and bounded away from $0$. Observe that $J'(t_n u_n)(t_n u_n) = 0$,  $(p^{1/(p-2)} t_nu_n)$  is bounded in $H^1(\R^d)$, and
$\lim_{n\to\infty}\sup_{y\in\R^d}\int_{B(y,1)}|p^{1/(p-2)}t_nu_n|^2\,dx=0.$ Then
 $$J(p^{1/(p-2)} t_nu_n)=-\frac{p^{2/(p-2)}}{2}(|\nabla (t_nu_n)|_2^2+|t_nu_n|_2^2)\leq -c<0$$
 for a constant $c>0$, which contradicts the statement  $\liminf_{n\to\infty}J(p^{1/(p-2)} t_nu_n)\geq 0$.
\end{remark}

In what follows $\lesssim$ denotes the inequality up to a multiplicative constant. Moreover $C > 0$ denotes a generic constant that may vary from one line to another.

The paper is organized as follows. In Section  \ref{lions} we prove Lions' type Lemma \ref{lem:vanishing} and in
Section \ref{th1} we present our variational approach and prove Theorem \ref{thm:1}.

\section{A variant of Lions' lemma}\label{lions}

We need the following lifting lemma.

\begin{lemma}[{\cite[Lemma 2.1]{maris}}] \label{lem:lifting}
Let $u \in \cX$ satisfy $m \leq |1+u(x)| \leq 2$ a.e. on $\R^d$ for some $m \in (0,2)$. Then, there are functions $\rho, \theta : \R^d \rightarrow \R$ such that $\rho -1 \in H^1 (\R^d)$, $\theta \in \cD^{1,2} (\R^d)$ and 
$$
1+u = \rho e^{i\theta} \quad \mbox{a.e. on } \R^d
$$
and
$$
\cQ(u) = - \frac12 \int_{\R^d} (\rho^2-1) \frac{\partial \theta}{\partial x_1} \, dx.
$$
\end{lemma}

Now, in order to prove Lemma \ref{lem:vanishing}, we recall the regularization procedure introduced in \cite{Gerard} and developed in \cite{maris}, which holds also in our more general setting.
For fixed $u\in \mathcal{X}$ and $h>0$, define
$$G(v) :=G^u_{h}(v):=E_{GL}(v)+\frac{1}{2 h^2}\int_{\R^d} \varphi(|v-u|^2) \, dx.$$
The next lemma gives the properties of functions that minimize $G$ in the space $$H^1_u(\R^d):=\{v\in \mathcal{X} : v-u\in H^1(\R^d)\}.$$

\begin{lemma}[{\cite[Lemma 3.1]{maris}}]\label{lem3.1}
Assume $d\ge3$. Then the following properties hold.
\begin{itemize}
\item[(i)] The functional $G$ has a minimizer in $H^1_u(\R^d)$.
\item[(ii)] Let $v_h$ be a minimizer of $G$ in $H^1_u(\R^d)$. There exists a constant $C_1 > 0$, depending only on d, such that $v_h$ satisfies
\begin{equation*}
E_{GL}(v_h)\le E_{GL}(u)
\end{equation*}
\begin{equation}\label{3.2}
|v_h-u|^2_2\le h^2E_{GL}(u)+C_1\bigl(E_{GL}(u)\bigr)^{1+2/d}h^{4/d}
\end{equation}
\begin{equation*}
\int_{\R^d}\bigl|(\varphi^2(|1+u|)-1)^2-(\varphi^2(|1+v_h|)-1)^2\bigr| \, dx  \lesssim hE_{GL}(u)
\end{equation*}
\begin{equation}\label{3.4}
|\cQ(u)-\cQ(v_h)| \lesssim \left(h^2+(E_{GL}(u))^{2/d}h^{4/d}\right)^{1/2}E_{GL}(u)
\end{equation}
\item[(iii)] For $z\in\mathbb C$, denote $H(z)=(\varphi^2(|z+1|)-1)\varphi(|z+1|)\varphi'(|z+1|)\frac{z+1}{|z+1|}$ if $z\neq-1$ and $H(-1) = 0$. Then any minimizer $v_h$ of $G$ in $H^1_u(\R^d)$ satisfies the equation
\begin{equation*}
-\Delta v_h+H(v_h)+\frac{1}{h^2}\varphi'(|v_h-u|^2)(v_h-u)=0\quad \text{in}\,\, \mathcal{D}'(\R^d).
\end{equation*}
Moreover, $v_h \in W^{2,p}_{\loc} (\R^d)$ for $p\in[1,\infty)$; in particular, $v_h\in \cC^{1,\alpha}_{\loc} (\R^d)$ for $\alpha\in[0, 1)$.
\item[(iv)] For any $h > 0$, $\delta > 0$, there exists a constant $K=K(d, h, \delta) > 0$ such that for any $u\in\mathcal{X}$ with $E_{GL}(u) \le K$ and for any minimizer $v_h$ of $G$ in $H^1_u(\R^d)$, there holds
\begin{equation*}
1-\delta<|1+v_h(x)|<1+\delta\qquad\text{for}\,\, x \in \R^d.
\end{equation*}
\end{itemize}
\end{lemma}

\begin{lemma} \label{lem:est1}
The assumption (F1) implies that for every $\varepsilon > 0$ there is $\delta = \delta(\varepsilon)$ such that
$$
\left| V(s^2) - \frac{1}{2} \left( s^2 - 1 \right)^2 \right| \leq \varepsilon \left( s^2 - 1 \right)^2
$$
for any $1-\delta \leq s \leq 1 + \delta$.
\end{lemma}

\begin{proof}
Note that, by (F1) and Taylor's expansion theorem, we get that
$$
V(\tau) = \frac12 V''(1) (\tau-1)^2 + (\tau-1)^2 \xi_\tau,
$$
where $\xi_\tau \to 0$ as $\tau \to 1$. Then
$$
\left| V(s^2) - \frac12 V''(1) (s^2-1)^2 \right| = (s^2-1)^2 | \xi_{s^2}|.
$$
For any $\varepsilon > 0$ we find $\delta> 0$ such that $| \xi_{s^2}| < \varepsilon$ for $1-\delta \leq s \leq 1 + \delta$, and we conclude.
\end{proof}

Let $2 < q < 2^*$ and $\delta > 0$. We observe that (F1) and (F2) imply that for every $\varepsilon > 0$, there is $C_\varepsilon > 0$ such that the following inequality holds
\begin{equation}\label{eq:est3}
		\left| V(s^2) - \frac12 \left( \varphi^2 (s)-1\right)^2 \right| \leq \varepsilon \left| s-1\right|^{2^{*}} + C_\varepsilon \left| s-1\right|^{q}.
\end{equation}
for all $s \geq 0$ such that $|s-1| > \delta$.

Before giving the proof of Lemma \ref{lem:vanishing} we introduce the following notation
$$
| u |_{p, A} := |u \chi_A |_p = \left( \int_A |u|^p \, dx \right)^{1/p}
$$
for $u \in L^p (\R^d)$ and any measurable set $A \subset \R^d$. Moreover, for any $u \in L^1_{\mathrm{loc}} (\R^d)$ and any measurable set $A \subset \R^d$ of finite measure we define the {\em integral average}
$$
m(u, A) := \fint_{A} u \, dx := \frac{1}{|A|} \int_A u \, dx.
$$

\begin{altproof}{Lemma \ref{lem:vanishing}}
Suppose that \eqref{eq:vanishing2} holds. We will show that
the following condition holds
\begin{equation}\label{eq:vanishing}
\lim_{n\to\infty}\sup_{y\in\R^d}\int_{B(y,1)}|u_n - m(u_n, B(y,1))|^2+\left(\varphi^2(|1+u_n|) - 1\right)^2\,dx=0.
\end{equation}
Observe that
\begin{align*}
\int_{B(y,1)} |m(u_n(B(y,1))|^2 \, dx &= |B(0,1)| |m(u_n(B(y,1))|^2 \\
&= \frac{1}{|B(0,1)|} \left| \int_{B(y,1)} u_n(x) \, dx \right|^2 \lesssim \int_{B(y,1)} |u_n|^2 \, dx.
\end{align*}
Hence
$$
\sup_{y\in\R^d} \int_{B(y,1)} |m(u_n(B(y,1))|^2 \, dx \lesssim \sup_{y\in\R^d} \int_{B(y,1)} |u_n|^2 \, dx \to 0
$$
thanks to \eqref{eq:vanishing2}. Thus
\begin{align*}
\sup_{y\in\R^d} \int_{B(y,1)} |u_n - m(u_n(B(y,1))|^2 +\left(\varphi^2(|1+u_n|) - 1\right)^2 \, dx &\leq \\
\sup_{y\in\R^d} \int_{B(y,1)} 2|u_n|^2 + 2|m(u_n(B(y,1))|^2 +\left(\varphi^2(|1+u_n|) - 1\right)^2 \, dx &\to 0
\end{align*}
and \eqref{eq:vanishing} hold.

Now, following the arguments in \cite[Lemma 3.2]{maris}, and for the reader’s convenience -- especially since we work under more general assumptions (F2) -- we provide the details below.

For the simplicity of notation, set $m_n(x) := m(u_n,B(x,1))$. We define
		$$
		h_n := \max \left\{ \left( \sup_{y \in \R^d} |u_n-m_n|_{2,B(y,1)} \right)^{\frac{1}{d+2}}, \left( \sup_{y \in \R^d} | H(m_n(y)) | \right)^{\frac{1}{d}} \right\},
		$$
		where $H$ is defined in Lemma \ref{lem3.1}(iii). Observe that for any $z \in \C$ we have
		\begin{equation}\label{3.12}
		|H(z)| \lesssim \left| \varphi^2 (|z+1|)-1 \right|.
		\end{equation}
		Note that $\sup_{y \in \R^d} |u_n-m_n(y)|_{2,B(y,1)} \to 0^+$ thanks to \eqref{eq:vanishing}. 
		Then \eqref{eq:vanishing} and \eqref{3.12} imply that
		$$
		\sup_{y \in \R^d} |H(u_n)|_{2,B(y,1)}^2 \lesssim \sup_{y \in \R^d} \int_{B(y,1)} \left( \varphi^2(|1+u_n|) - 1 \right)^2 \, dx \to 0^+.
		$$
		Since $H$ is Lipschitz-continuous on $\C$
		$$
		\sup_{y\in\R^d} |H(u_n)-H(m_n(y))|_{2,B(y,1)} \lesssim \sup_{y\in\R^d} |u_n-m_n(y)|_{2, B(y,1)} \to 0^+
		$$
		again thanks to \eqref{eq:vanishing}. Hence $\sup_{y\in\R^d} |H(m_n(y))|_{2,B(y,1)} \to 0$ and therefore $h_n \to 0^+$.

		Hence, we may assume that $0 < h_n < 1$ for any $n \geq 1$. Let $v_n$ be the minimizer of $G_{h_n}^{u_n}$ given by Lemma \ref{lem3.1}(ii), it satisfies 
		$$
		-\Delta v_n + H(v_n) + \frac{1}{h_n^2} \varphi' (|v_n-u_n|^2) (v_n - u_n) = 0
		$$
		in the sense of distributions. 
		
In view of the elliptic regularity, we follow arguments presented in {\em Step 2} and {\em Step 3} from \cite[Lemma 3.2]{maris} and we obtain that $v_n$ are  locally uniformly H\"older continuous  with exponent $\frac12$, namely there are $R > 0$, $C_* > 0$ such that for any $n$ there holds $|v_n(x)-v_n(y)| \leq C_* |x-y|^{1/2}$ if $|x-y| < R$. Now, put $$\delta_n := \left| |v_n+1|-1 \right|_\infty.$$ 
		 Using Lipschitz-continuity of $z \mapsto (\varphi^2(|1+z|)-1)^2$, \eqref{3.2} and boundedness of $E_{GL}(u_n)$, we get
\begin{align*}
&\quad \int_{B(x,1)} \left| \left( \varphi^2 (|1+v_n(y)|)-1 \right)^2 - \left( \varphi^2 (|1+u_n(y)|)-1 \right)^2 \right| \, dy \\
&\lesssim \int_{B(x,1)} |v_n(y)-u_n(y)| \, dy \lesssim |v_n-u_n|_{2,B(x,1)} \lesssim |v_n-u_n|_2 \lesssim h_n^{2/d}.
\end{align*}
Since $d \geq 3$, and using \eqref{eq:vanishing} we get that 
\begin{equation*}
\sup_{x \in \R^d} \int_{B(x,1)} \left( \varphi^2 (|1+v_n(y)|)-1 \right)^2 \, dy \to 0.
\end{equation*}
Choose $x_n \in \R^d$ such that
$$
\left| |v_n(x_n)+1|-1 \right| \geq \frac12 \delta_n.
$$
Then - from the local uniform H\"older continuity, $|v_n(x_n) - v_n(x)| \leq C_* |x_n - x|^{1/2}$ if $x \in B(x_n, R)$. Without loss of generality we may assume that $R < 1$. Set 
$$
r_n := \min \left\{ R, \left( \frac{\delta_n}{4 C_*} \right)^2 \right\}.
$$
Then, if $|x_n - x| < r_n$, we get
$$
\Big| \left| |v_n(x_n)+1|-1 \right| -  \left| |v_n(x)+1|-1 \right| \Big| \leq |v_n(x_n)-v(x)| \leq C_* |x_n-x|^{1/2} < C_* r_n^{1/2} \leq \frac{\delta_n}{4}.
$$
Hence
$$
\left| |v_n(x)+1|-1 \right| \geq \frac14 \delta_n
$$
if $x \in B(x_n, r_n)$. Then
		\begin{align*}
		\int_{B(x_n,1)} \left( \varphi^2 (|1+v_n(y)|)-1 \right)^2 \, dy &\geq \int_{B(x_n, r_n)} \left( \varphi^2 (|1+v_n(y)|)-1 \right)^2 \, dy \\
		&\geq \int_{B(x_n, r_n)} \eta \left( \frac{\delta_n}{4} \right) \,dy  \gtrsim \eta \left( \frac{\delta_n}{4} \right) r_n^d,
		\end{align*}
		where $\eta$ is given by
\begin{equation}\label{3.30}
\eta(s):=\inf\{(\varphi^2(\tau)-1)^2 : |1-\tau| \geq s \}.
\end{equation}		
Hence $\delta_n \to 0$.

		Now we claim that 
		$$
		E_{GL}(v_n) + c\cQ(v_n) \geq 0
		$$
		for large $n$. Fix $n$ sufficiently large so that $1-\delta \leq |v_n+1| \leq 1 + \delta$ and $\delta < 1$ (note that $\delta$ doesn't depend on $n$). From Lemma \ref{lem:lifting} there are real valued functions $\rho_n, \theta_n$ such that $\rho_n -1 \in H^1 (\R^d)$, $\theta_n \in \cD^{1,2}(\R^d)$ such that $1+v_n = \rho_n e^{i\theta_n}$ and
		$$
		\cQ(v_n) = - \frac12 \int_{\R^d} (\rho_n^2-1) \frac{\partial \theta_n}{\partial x_1} \, dx.
		$$
		Thus
		\begin{align*}
		\sqrt{2}(1-\delta)|\cQ(v_n)| &\leq \frac{\sqrt{2}}{2} (1-\delta) \left| \frac{\partial \theta_n}{\partial x_1} \right|_2 | \rho_n^2 - 1 |_2 \leq \frac12 (1-\delta)^2 \left| \frac{\partial \theta_n}{\partial x_1} \right|_2^2 + \frac14 | \rho_n^2 - 1 |_2^2 \\
		&\leq \frac12 \int_{\R^d} \rho_n^2 |\nabla \theta_n|^2 + \frac14 (\rho_n^2 - 1)^2 \, dx = E_{GL}(v_n).
		\end{align*}
		Hence, for sufficiently large $n$ and $|c| < \sqrt{2}$ we get
		$$
		E_{GL}(v_n) + c\cQ(v_n) \geq 0.
		$$
		Now \eqref{3.4} implies that
		$$
		\left| \cQ(u_n) - \cQ(v_n) \right| \lesssim \left( h_n^2 + (E_{GL}(u_n))^{2/d} h_n^{4/d} \right)^{1/2} E_{GL}(u_n) \lesssim \left( h_n^2 + M^{2/d} h_n^{4/d} \right)^{1/2} M \to 0,
		$$
		where, $M := \sup_n E_{GL}(u_n)$. Hence,
		\begin{align*}
		&\quad E_c(u_n) = E_{GL}(u_n) + c \cQ(u_n) + \frac12 \int_{\R^d} \left( V(|1+u_n|^2) - \frac12 \left( \varphi^2 (|1+u_n|-1\right)^2 \right) \, dx \\
		&\geq \underbrace{E_{GL}(v_n) + c\cQ(v_n)}_{\geq 0} + \underbrace{c (\cQ(u_n)-\cQ(v_n))}_{\to 0} - \frac12 \int_{\R^d} \left| V(|1+u_n|^2) - \frac12 \left( \varphi^2 (|1+u_n|-1\right)^2 \right| \, dx.
		\end{align*}
		It is sufficient to show that 
		$$
		\int_{\R^d} \left| V(|1+u_n|^2) - \frac12 \left( \varphi^2 (|1+u_n|)-1\right)^2 \right| \, dx \to 0.
		$$
		Take any $\varepsilon > 0$. Lemma \ref{lem:est1} implies that
		$$
		\left| V(|1+z|^2) - \frac12 \left( |1+z|^2-1\right)^2 \right| \leq \varepsilon \left( |1+z|^2-1\right)^2
		$$
		for $\left| |1+z|-1 \right| \leq \delta = \delta(\varepsilon)$. Without loss of generality we may assume that $\delta \leq 1$ and therefore
		$$
		\left| V(|1+z|^2) - \frac12 \left( \varphi^2 (|1+z|)-1\right)^2 \right| \leq \varepsilon \left( \varphi^2 (|1+z|)-1\right)^2.
		$$
		for $\left| |1+z|-1 \right| \leq \delta$. Thus
		\begin{align*}
		&\quad \int_{\{ ||1+u_n|-1| \leq \delta\}} \left| V(|1+u_n|^2) - \frac12 \left( \varphi^2 (|1+u_n|)-1\right)^2 \right| \, dx \\
		&\leq \varepsilon \int_{\{ ||1+u_n|-1| \leq \delta\}} \left( \varphi^2 (|1+u_n|)-1\right)^2 \, dx \leq \varepsilon M.
		\end{align*}
		Define $w_n := ||1+u_n|-1|$. Then $w_n \leq |u_n|$, $w_n \in \cD^{1,2}(\R^d)$ and
		$$
		\int_{\R^d} |\nabla w_n|^2 \, dx \lesssim 1.
		$$
		In particular, from the continuity of the embedding $\cD^{1,2} (\R^d) \subset L^{2^*} (\R^d)$, $(w_n)$ is bounded in $L^{2^*}(\R^d)$.
		Note that from \eqref{eq:est3},
		$$
		\left| V(|1+z|^2) - \frac12 \left( \varphi^2 (|1+z|)-1\right)^2 \right| \leq \varepsilon \left| |1+z|-1\right|^{2^{*}} + C_\varepsilon \left| |1+z|-1\right|^{q}.
		$$
		Thus
		\begin{equation} \label{vphi}
		\begin{aligned}
		&\quad \int_{\{ ||1+u_n|-1| > \delta\}} \left| V(|1+u_n|^2) - \frac12 \left( \varphi^2 (|1+u_n|)-1\right)^2 \right| \, dx \\
		&\leq \varepsilon \int_{\{w_n > \delta\}} |w_n|^{2^*} \, dx + C_\varepsilon \int_{\{w_n > \delta\}} |w_n|^{q} \, dx \\
		&\leq \varepsilon |w_n|_{2^*}^{2^*} + C_\varepsilon \left( \int_{\{w_n > \delta\}} |w_n|^{2^*} \, dx \right)^{q/2^*} | \{ x \ : \  w_n(x) > \delta \}|^{1 - q/2^*} \\
		&\leq \varepsilon |w_n|_{2^*}^{2^*} + C_\varepsilon C_S^{q} |\nabla w_n|_2^{q} | \{ x \ : \  w_n(x) > \delta \}|^{1 - q/2^*} \\
		&\leq \varepsilon \underbrace{|w_n|_{2^*}^{2^*}}_{\mathrm{bounded}} + C_\varepsilon C_S^{q} M^{q/2} | \{ x \ : \  w_n(x) > \delta \}|^{1 - q/2^*},
		\end{aligned}\end{equation}
		where we note that $2 < q < 2^*$. We will show that for any $\delta > 0$
		$$
		| \{ x \ : \  w_n(x) > \delta \}| \to 0 \quad \mbox{as } n\to\infty.
		$$
		By contradiction, assume that (up to a subsequence), 
		$$
		| \{ x \ : \  w_n(x) > \delta_0 \}| \geq \gamma > 0
		$$
		for some $\delta_0$. Since $|\nabla w_n|_2$ stays bounded, then - from Lieb's Lemma (\cite[Lemma 6]{Lieb}, \cite[Lemma 2.2]{BrezisLieb})
		$$
		\beta := \inf_n \left| \{ x \ : \ w_n(x) > \delta_0/2 \} \cap B(y_n, 1) \right| > 0
		$$
		for some $y_n$. However, then if $w_n(x) > \delta_0/2$, then 
		$$
		\left( \varphi^2 (|1+u_n|)-1\right)^2 \geq \eta(\delta_0/2) > 0, 
		$$
		where $\eta$ is defined by \eqref{3.30}. Hence
		$$
		\int_{B(y_n, 1)} \left( \varphi^2 (|1+u_n|)-1\right)^2 \, dx \geq \eta(\delta_0/2) \beta > 0
		$$ 
		which is a contradiction with \eqref{eq:vanishing}. This, by \eqref{vphi}, proves that
		$$
		\int_{\{ ||1+u_n|-1| > \delta\}} \left| V(|1+u_n|^2) - \frac12 \left( \varphi^2 (|1+u_n|)-1\right)^2 \right| \, dx \to 0
		$$
		for any $\delta>0$ and therefore
		$$
		\int_{\R^d} \left| V(|1+u_n|^2) - \frac12 \left( \varphi^2 (|1+u_n|)-1\right)^2 \right| \, dx \to 0.
		$$
		and the proof is completed.
\end{altproof}

\section{Variational approach and proof of Theorem \ref{thm:1}} \label{th1}

Before giving the proof of Theorem \ref{thm:1}, we state some preliminary results.
Let $\sigma>0$ and define $u_{1,\sigma}(x):=u(x_1, y/\sigma)$ so that the Lagrangian has the following form
$$E_c(u_{1,\sigma})=\frac12\sigma^{d-3}\int_{\R^d}|\nabla_y u|^2\,dx+\frac12\sigma^{d-1}\int_{\mathbb R^d} |\partial_{x_1} u|^2+V(|1+u|^2)\,dx+\sigma^{d-1}c\cQ(u).$$
We will use the {\em Poho\v{z}aev constraint} of the form
\begin{equation*}
\mathcal M=\big\{u \in \cX\setminus\{0\} \ : \ P_c(u)=0\big\},
\end{equation*}
where
\begin{equation*}\begin{aligned}
P_c(u)&:=\frac{1}{d-1}\frac{\partial}{\partial \sigma}E_c(u_{1,\sigma})|_{\sigma=1}\\
&=\frac{d-3}{2(d-1)}\int_{\R^d}|\nabla_{y} u|^2\,dx+\frac{1}{2}\int_{\R^d}|\partial_{x_1} u|^2+V(|1+u|^2)\,dx+c\cQ(u)\\
&=\frac{d-3}{2(d-1)}\int_{\R^d}|\nabla_{y} u|^2\,dx-L(u),
\end{aligned}\end{equation*}
and
\begin{equation}\label{defL}
L(u):=-\biggl[\frac{1}{2}\int_{\mathbb R^d}|\partial_{x_1} u|^2+V(|1+ u|^2)\,dx+c\mathcal{Q}(u)\biggr]=-\biggl[E_c(u)-\frac{1}{2}\int_{\mathbb R^d}|\nabla_y u|^2\,dx\biggr].
\end{equation}
Observe that
\begin{equation*}
\begin{aligned}
P_c(u_{1,\sigma})
&=\frac{d-3}{2(d-1)}\int_{\R^d}|\nabla_{y} u_{1,\sigma}|^2\,dx+\!\frac{1}{2}\!\int_{\R^d}|\partial_{x_1} u_{1,\sigma}|^2\!+\!V(|1+u_{1,\sigma}|^2)\,dx+c\cQ(u_{1,\sigma})\\
&=\sigma^{d-3}\biggl[\frac{d-3}{2(d-1)}\int_{\mathbb R^d}|\nabla_{y} u|^2\,dx-\sigma^2L(u)\biggr].
\end{aligned}\end{equation*}
We will define $m_{\mathcal{M}}(u):=u_{1,\sigma}$, the projection of $u$ onto $\mathcal{M}$ as follows.

\begin{lemma}\label{proj}
Let $d\geq 4$ and
	$$\mathcal{U}:=\{u\in \mathcal{X} : L(u)>0\}.$$
	Then $\mathcal{U}\neq\emptyset$ and for any $u\in \mathcal{U}$ there exist a unique $\sigma=\sigma_u>0$ such that
	$$\sigma_u:=\biggl(\frac{(d-3)\int_{\mathbb R^d}|\nabla_{y} u|^2\,dx}{2(d-1)L(u)}\biggr)^{1/2}$$
	and $u_{1,\sigma}\in \mathcal{M}$. Thus $\mathcal{U}\supset\mathcal{M}\neq\emptyset$ and $m_{\mathcal{M}}:\mathcal{U}\to \mathcal{M}$ is well-defined.
\end{lemma}

\begin{proof} 
	We only need to prove that $\mathcal{U}\neq\emptyset$, which is equivalent to find $w\in\mathcal{X}\setminus\{0\}$ such that $E_c(w)<0$.
	From  \cite[Lemma 4.4]{maris}, there exists a continuous map $R\mapsto v_R$ from $[2,\infty)$ to $H^1(\mathbb R^d)$, such that $E_c(v_R)<0$ for $R$ sufficiently large.
\end{proof}

Moreover we observe that arguing as  in \cite[Lemma 4.3]{maris} under assumptions (F1)--(F2) we show that for sufficiently small $\eps>0$ there is $M>0$ such that 
\begin{equation}\label{eqCompEnergies}
E_c(u)\geq \cE(u)-c|\cQ(u)|\geq \eps E_{GL}(u)
\end{equation}
for any $u\in\cX$ such that $E_{GL}(u)\leq M$.

Therefore we may define 
\begin{equation*}
T_c:=\inf_{\mathcal{M}}E_c>0.
\end{equation*}
Note that, for any $u\in \mathcal{M}$, then the Lagrangian becomes
\begin{equation}\label{encon}
E_c(u)=\frac{1}{d-1}\int_{\mathbb R^d}|\nabla_{y} u|^2 \, dx.
\end{equation}
Thus, if $u\in\cM\cap \cU$, then from the definition of $L$ in \eqref{defL}, we have
\begin{equation*}
\frac{d-3}{2(d-1)}\int_{\mathbb R^d}|\nabla_{y} u|^2\,dx-L(u)=0 \iff L(u)=\frac{d-3}{2(d-1)} E_c(u).
\end{equation*}

\begin{proposition}\label{sob}
	Let $d\geq 4$. For every $u\in \mathcal{U}$, it holds
	$$\int_{\mathbb R^d}|\nabla_{y} u|^2\,dx\ge K_0 T_c^{2/(d-1)}L(u)^{(d-3)/(d-1)},$$
	where
	$$K_0:=(d-1)\biggl(\frac{2}{d-3}\biggr)^{(d-3)/(d-1)}$$
\end{proposition}

\begin{proof}
Take $u\in \mathcal{U}$ and note that by Lemma \ref{proj}, \eqref{encon} and by the definition of $T_c$, we get
$$\begin{aligned}
T_c&\le E_c(u_{1,\sigma_u})=\frac{1}{d-1}\int_{\mathbb R^d}|\nabla_{y} u_{1,\sigma_u}|^2\,dx=\frac{1}{d-1}\sigma_u^{d-3}\int_{\mathbb R^d}|\nabla_{y} u|^2\,dx\\
&=\frac{1}{d-1} \biggl(\frac{(d-3)\int_{\mathbb R^d}|\nabla_{y} u|^2\,dx}{2(d-1)L(u)}\biggr)^{(d-3)/2}\int_{\mathbb R^d}|\nabla_{y} u|^2\,dx\\
& =\frac{(d-3)^{(d-3)/2}}{(d-1)^{(d-1)/2}}\frac{\biggl(\int_{\mathbb R^d}|\nabla_{y} u|^2\,dx\biggr)^{(d-1)/2}}{(2L(u))^{(d-3)/2}},
	\end{aligned}$$
	which completes the proof.
\end{proof}

Below we use the Lions' type lemma (Lemma \ref{lem:vanishing}) to find a non-zero weak limit point of a minimizing sequence, which allows to use our variational approach.

\begin{lemma}\label{nontrivial}
	Let $d\geq 4$ and $(u_n)\subset \mathcal{M}$ be a sequence such that $(E_{GL}(u_n))$ is bounded. Then there is a sequence $(x_n)\subset\R^d$ and $u\in \cX \setminus\{0\}$ such that $u_n(\cdot+x_n)\weakto u$ in $\cD^{1,2}(\mathbb R^d)$.
\end{lemma}
\begin{proof}
In a similar way as in \cite[Lemma 5.4, step 2]{maris}, let $\sigma:=\sqrt{\frac{2(d-1)}{d-3}}$, and $\widetilde u_n:=(u_n)_{1,\sigma}$. Using the definition of $L$ in \eqref{defL} we obtain that
	$$\begin{aligned}
	0&=P_c(u_n):=\frac{d-3}{2(d-1)}\int_{\R^d}|\nabla_{y}  u_n|^2\,dx-L(u_n)
	=\frac{d-3}{2(d-1)}\sigma^{3-d}\int_{\R^d}|\nabla_{y}  \widetilde u_n|^2\,dx-\sigma^{1-d}L(\widetilde u_n)\\
	&=\sigma^{1-d}\biggl[\frac{d-3}{2(d-1)}\sigma^{2}\int_{\R^d}|\nabla_{y}  \widetilde u_n|^2\,dx-L(\widetilde u_n)\biggr]=\sigma^{1-d}\biggl[\int_{\R^d}|\nabla_{y}  \widetilde u_n|^2\,dx-L(\widetilde u_n)\biggr]\\
 & =\sigma^{1-d}\biggl[\frac{1}{2}\int_{\R^d}|\nabla_{y}  \widetilde u_n|^2\,dx+E_c(\widetilde u_n)\biggr].
	\end{aligned}$$
Therefore 
	$$\frac{1}{2}\int_{\R^d}|\nabla_{y}  \widetilde u_n|^2\,dx+E_c(\widetilde u_n)=0.$$
From Proposition \ref{sob} we infer that $\int_{\R^d}|\nabla_{y}  \widetilde u_n|^2\,dx=\sigma^{d-3}\int_{\R^d}|\nabla_{y} u_n|^2\,dx$ is positive and away from $0$. If $( \widetilde u_n)$ satisfies \eqref{eq:vanishing2} then, in view of Lemma \ref{lem:vanishing}, since also $(E_{GL}(\wt u_n))$ is bounded, we get a contradiction, thus necessarily \eqref{eq:vanishing2} does not hold. Therefore there is $\delta>0$ and $(x_n)\subset\R^d$ such that
 \begin{equation}\label{eq:lemvan}
 \int_{B(x_n,1)}|\wt u_n|^2+\left(\varphi^2(|1+\wt u_n|) - 1\right)^2\,dx\ge \delta.
 \end{equation}
Note that there is $u\in \cX$ such that $u_n(\cdot+x_n)\weakto u$ in $\cD^{1,2}(\mathbb R^d)$. In view of \eqref{eq:lemvan} and since $\cD^{1,2}(\mathbb{R}^d)$ is locally compactly embedded into $L^2_{\loc} (\R^d)$ we infer that $u\neq 0$.
\end{proof}

Now we have all the ingredients to prove Theorem \ref{thm:1}.

\begin{altproof}{Theorem \ref{thm:1}}
Concerning (a), we take $S_c:=K_0^{(d-1)/(d-3)} T_c^{2/(d-3)}>0$. Clearly, \eqref{eq:ineq} holds for any $u\in\cX$ such that $L(u)\leq 0$. If $L(u)>0$, i.e. $u\in\cU$, then \eqref{eq:ineq} follows from Proposition \ref{sob}.

In order to prove (b)--(c), let $(u_n)\subset \cX$ be a sequence such that 
$$
\frac{ \left(\int_{\R^d}|\nabla_y u_n|^2\, dx \right)^{\frac{d-1}{d-3}}}{\frac{1}{2}\int_{\R^d}|\partial_{x_1} u_n|^2 +V(|1+u_n|^2)\,dx+c\cQ(u_n) } \to -S_c
$$
as $n\to\infty$, so that $u_n\in\cU$ for $n$ large. Let $\wt u_n:=m_{\cM}(u_n)\in\cM$ and observe that 
\begin{equation}\label{totc}
E_c(\wt u_n)=\frac{1}{d-1}\int_{\R^d}|\nabla_y \wt u_n|^2\,dx=
\frac{1}{d-1}\sigma_{u_n}^{d-3}\int_{\R^d}|\nabla_y  u_n|^2\,dx\to T_c.
\end{equation}
Next, we claim that $(E_{GL}(\wt u_n))$ is bounded. In particular, recall that
$$0=P_c(\wt u_n)=\frac{d-3}{2(d-1)}\int_{\R^d}|\nabla_{y} \wt u_n|^2\,dx+\frac{1}{2}\int_{\R^d}|\partial_{x_1} \wt u_n|^2+V(|1+\wt u_n|^2)\,dx+c\cQ(\wt u_n)$$
and, by \eqref{totc} we deduce the boundedness of $\nabla_{y} \wt u_n$ in $L^2(\mathbb R^d)$.
Now we show the boundedness of $\partial_{x_1}\wt u_n, \left( \varphi^2(|1+\wt u_n|) - 1 \right)$ in $L^2(\mathbb R^d)$.
Observe that, from \eqref{eqCompEnergies} we find $\delta>0$ such that
\begin{equation}\label{eq:estbelow}
E_c(u)\geq \eps \delta
\end{equation}
for $u\in\cX$ such that $E_{GL}(u)=\delta$.
Moreover
 we find a unique sequence $(\wt\sigma_n)\subset (0,+\infty)$ such that
$$\delta=E_{GL}(\wt u_n(x_1, y/\wt\sigma_n)) =\wt\sigma_n^{d-1}\int_{\mathbb R^d} \frac12 |\partial_{x_1}\wt u_n|^2+ \frac14 \left( \varphi^2(|1+\wt u_n|) - 1 \right)^2 \, dx+\wt\sigma_n^{d-3}\int_{\mathbb R^d} \frac12 |\nabla_{y}\wt u_n|^2 \, dx.$$
Note that, if $|\partial_{x_1}\wt u_n|_2\to\infty$ or $\int_{\R^d} \left( \varphi^2(|1+\wt u_n|) - 1 \right) \, dx\to\infty$, then
 $\wt\sigma_n\to0$ and
 $$ E_c(\wt u_n(x_1, y/\wt\sigma_n))=\frac{1}{d-1}\wt\sigma_n^{d-3}\int_{\mathbb R^d}|\nabla_{y}\wt u_n|^2 \, dx\to 0,\qquad n\to\infty,$$
 which contradicts \eqref{eq:estbelow}.

Now, in view of Lemma \ref{nontrivial}, there exists a sequence $(x_n)\subset \R^d$ and $\wt u\in \cX \setminus\{0\}$ such that $\wt{u}_n(\cdot+x_n)\weakto\wt{u}$ in $\cD^{1,2}(\mathbb R^d)$ and $\wt u_n(\cdot+x_n)\to\wt u$ a.e. in $\mathbb R^d$. 
For simplicity of notation, in what follows we may assume that $x_n=0$ for every $n\in\mathbb \N$.

Take $v\in \cC_0^\infty(\mathbb R^d)$ and $t\in \mathbb R$ sufficiently small such that $L(\wt u_n+tv)>0$, namely $(\wt u_n+tv)\subset\mathcal{U}$. Observe that
\begin{equation*}\begin{aligned}
L(\wt u_n+tv)&-L(\wt u_n)=-\int_{\mathbb R^d} c[\langle i\partial_{x_1}(\wt u_n+tv), \wt u_n+tv\rangle-\langle i\partial_{x_1}\wt u_n, \wt u_n\rangle]\\
&+\frac{1}{2}\Bigl[V(|1+\wt u_n+tv|^2)-V(|1+\wt u_n|^2)+|\partial_{x_1}(\wt u_n+tv)|^2-|\partial_{x_1}\wt u_n|^2\Bigr]\,dx,
\end{aligned}\end{equation*}
and, passing to a subsequence,
\begin{eqnarray*}
\int_{\mathbb R^d} \langle i\partial_{x_1}(\wt u_n+tv), \wt u_n+tv\rangle-\langle i\partial_{x_1}\wt u_n, \wt u_n\rangle \, dx&=&t\int_{\mathbb R^d} \langle i\partial_{x_1}(\wt u_n+tv), v\rangle+\langle i\partial_{x_1}v, \wt u_n\rangle \, dx\\
&\to& t\int_{\mathbb R^d} \langle i\partial_{x_1}(\wt u+tv), v\rangle+\langle i\partial_{x_1}v, \wt u\rangle \, dx.
\end{eqnarray*}
Moreover
\begin{eqnarray*}
\int_{\mathbb R^d} |\partial_{x_1}(\wt u_n+tv)|^2-|\partial_{x_1}\wt u_n|^2 \, dx&=&t\int_{\mathbb R^d} t|\partial_{x_1}v|^2+2\Re(\partial_{x_1} \wt u_n \partial_{x_1} v) \, dx\\
&\to& t\int_{\mathbb R^d} t|\partial_{x_1}v|^2+2\Re(\partial_{x_1} \wt u \partial_{x_1} v) \, dx.
\end{eqnarray*}
In view of (F1) and (F2), note that there is a constant $C>$ such that
$$|F(|1+s|^2)||1+s|\leq C\big(||1+s|-1|+|s|^{2^*-1}\big)$$
for $s\in\C$.  Then
by the mean value theorem, there is $\theta \in [0,1]$ such that
\[
\begin{aligned}
\big|V(|1+\wt u_n+tv|^2) - V(|1+\wt u_n|^2)\big| &\leq  2 \left| F(|1+\wt u_n+\theta tv|^2) \right|
|1+\wt u_n+\theta tv||tv|\\
&\lesssim \big(\big||1+\wt u_n+\theta tv|-1\big||tv|+|\wt u_n+\theta tv|^{2^*-1}|tv|\big)
\\
&\lesssim \big(\big||1+\wt u_n|-1\big||tv|+|tv|^2+(|\wt u_n|+|tv|)^{2^*-1}|tv|\big).
\end{aligned}
\] 
Since $(|1+\wt u_n|-1) \subset L^2(\R^d)$, $(\wt u_n) \subset L^{2^*}(\R^d)$ are bounded,
 we infer that the family $$\left\{ V(|1+\wt u_n+tv|^2) - V(|1+\wt u_n|^2) \right\}$$ is 
 uniformly integrable and by the Vitali's convergence theorem 
 \begin{align*}
 \int_{\R^d}V(|1+\wt u_n+tv|^2) - V(|1+\wt u_n|^2)\, dx &= \int_{\supp v} V(|1+\wt u_n+tv|^2) - V(|1+\wt u_n|^2)\, dx \\
 &\to \int_{\R^d}V(|1+\wt u+tv|^2) - V(|1+\wt u|^2)\, dx.
 \end{align*}
 Therefore
 $$\lim_{n\to\infty}\big(L(\wt u_n+tv)-L(\wt u_n)\big)=L(\wt u+tv)-L(\wt u).$$
Note that $(L(\wt u_n))$ is bounded since $( \wt u_n)\subset\cM$ and
$$E_c(\wt u_n)=\frac{1}{d-1}\int_{\mathbb R^d}|\nabla_{y} \wt u_n|^2\,dx=\frac{2}{d-3}L(\wt u_n).$$
Passing to a subsequence, let us define $A:=\lim_{n\to\infty} L(\wt u_n)$, hence $$\lim_{n\to\infty} L(\wt u_n+tv)=A+L(\wt u+tv)-L(\wt u).$$ Moreover 
$$\begin{aligned}
\lim_{n\to\infty}\frac{1}{t}&\biggl[L(\wt u_n+tv)^{(d-3)/(d-1)}-L(\wt u_n)^{(d-3)/(d-1)}\biggr]\\
&=\frac{1}{t}\biggl[(A+L(\wt u+tv)-L(\wt u))^{(d-3)/(d-1)}-A^{(d-3)/(d-1)}\biggr]\end{aligned}$$
Now we can pass to the limit for $t\to 0$, since $L$ is differentiable in the $v$-direction, and we  obtain
\begin{equation}\label{lund3}\begin{aligned}
\lim_{t\to0}&\lim_{n\to\infty}\frac{1}{t}\biggl[L(\wt u_n+tv)^{(d-3)/(d-1)}-L(\wt u_n)^{(d-3)/(d-1)}\biggr]\\
&=\frac{d-3}{d-1}(A+L(\wt u+tv)-L(\wt u))^{-2/(d-1)}(A+L(\wt u+tv)-L(\wt u))'\mid_{t=0}\\
&=\frac{d-3}{d-1}A^{-2/(d-1)} L'(\wt u)v.
\end{aligned}\end{equation}
Since $(\wt u_n)\subset\mathcal{M}$ and $(\wt u_n+tv)\subset\mathcal{U}$, we can apply Proposition \ref{sob} with $u=\wt u_n+tv$ and $u=\wt u_n$, so that we get
$$\frac{1}{t}\int_{\mathbb R^d} |\nabla_{y}(\wt u_n+tv)|^2-|\nabla_{y}\wt u_n|^2\,dx\ge \frac{1}{t} K_0 T_c^{2/(d-1)}\biggl[L(\wt u_n+tv)^{(d-3)/(d-1)}-L(\wt u_n)^{(d-3)/(d-1)}+o_n(1)\biggr]$$
Now, letting $n\to\infty$, and then letting $t\to 0$, in view of \eqref{lund3} we get
\begin{equation}\label{lim1}
\lim_{t\to0}\lim_{n\to\infty}\frac{1}{t}\int_{\mathbb R^d} |\nabla_{y}(\wt u_n+tv)|^2-|\nabla_{y}\wt u_n|^2\,dx\ge \frac{d-3}{d-1}K_0 T_c^{2/(d-1)}A^{-2/(d-1)} L'(\wt u)v.
\end{equation}
On the other hand, 
\begin{equation}\label{lim2}\begin{aligned}
\lim_{t\to0}&\lim_{n\to\infty} \frac{1}{t}\int_{\mathbb R^d}  |\nabla_{y}(\wt u_n+tv)|^2-|\nabla_{y}\wt u_n|^2\,dx\\
&=
\lim_{t\to0}\lim_{n\to\infty} \frac{1}{t}\int_{\mathbb R^d} 2\langle\nabla_{y} \wt u_n,\nabla_{y} tv\rangle+t^2|\nabla_{y}v|^2\,dx\\
&=\lim_{t\to0}\int_{\mathbb R^d} 2\langle\nabla_{y} \wt u,\nabla_{y} tv\rangle+t^2|\nabla_{y}v|^2\,dx=\int_{\mathbb R^d} 2\langle\nabla_{y} \wt u,\nabla_{y} v\rangle\,dx.
\end{aligned}\end{equation}
Using \eqref{lim1}, \eqref{lim2} and recalling the definitions of $A$ and $K_0$, we finally get
$$2\int_{\mathbb R^d} \langle\nabla_{y} \wt u,\nabla_{y} v\rangle\,dx\ge \frac{d-3}{d-1}K_0 T_c^{2/(d-1)}A^{-2/(d-1)} L'(\wt u)v=L'(\wt u)v.$$
Since $v$ is arbitrarily,  we get
$$2\int_{\mathbb R^d} \langle\nabla_{y} \wt u,\nabla_{y} v\rangle\,dx=L'(\wt u)v$$
for any $v\in\cC_0^{\infty}(\R^d)$,
which means that $\wt u$ is a weak solution to \eqref{eq} and (c) is proved.

We will show that $\wt u \in W^{2,q}_{\mathrm{loc}} (\R^d)$ for every $q < \infty$. Indeed, note that since $\wt u$ is a weak solution to \eqref{eq}, following \cite[Lemma 5.5]{maris} and introducing $\wt w := e^{\frac{-i c x_1}{2}} (1 + \wt u)$ we find that $\wt w \in H^1_{\mathrm{loc}} (\R^d)$ is a distributional solution to
\begin{align*}
-\Delta \wt w - \frac{c^4}{4} \wt w = F \left( |\wt w|^2 \right).
\end{align*}
From Brezis-Kato's result \cite{BrezisKato}, \cite[Lemma B.3]{Struwe}, we obtain that $\wt w \in W^{2,q}_{\mathrm{loc}} (\R^d)$ for every $q < \infty$. Then also $\wt u \in W^{2,q}_{\mathrm{loc}} (\R^d)$. This regularity is enough to show that $\wt u$ satisfies the Poho\v{z}aev identity, namely $\wt u \in \cM$.

Note that, from the proof above, it follows that $T_c$ and $S_c^{(d-3)/(d-1)}$ are attained, namely $T_c=E_c(\wt u)$ and $\int_{\R^d}|\nabla_y \wt u|^2 \, dx=S_c^{(d-3)/(d-1)}$. 
Indeed, since $(\wt u_n)\subset\mathcal{M}$ is a minimizing sequence for $T_c$, by the weak convergence $\wt u_n\rightharpoonup \wt u$ in $\cD^{1,2}(\mathbb R^d)$, we get
$$E_c(\wt u)=\frac{1}{d-1}\int_{\mathbb R^d}|\nabla_{y}\wt u|^2\,dx\le E_c(\wt u_n)\to T_c.$$
Moreover by the Poho\v{z}aev identity, $\wt u\in \cM$ and $$T_c\le E_c(\wt u),$$
thus the equality holds and $E_c(\wt u)=T_c$, so that also $\int_{\R^d}|\nabla_y \wt u|^2 \, dx=S_c^{(d-3)/(d-1)}$ and $L(\wt u)=-1$. Hence also $u_n\weakto \tu$.

 Finally, observe that if $u$ is a minimizer of \eqref{eq:ineq}, then by the above considerations, we get that $\tu:=m_{\cM}(u)$ is a solution and we complete the proof. 
\end{altproof}

\begin{remark}\label{rem:final}
Observe that in the proof of Theorem~\ref{thm:1}, we do not require any regularity properties and the bootstrap argument to solve~\eqref{eq}, which appears to be difficult to carry out under the general growth condition~(F2). However, we note that the solutions must belong to \( W^{2,q}_{\mathrm{loc}}(\mathbb{R}^d) \) for any \( q < \infty \) in view of Brezis-Kato's theorem, and this fact is only crucial for proving that a solution lies in \( \mathcal{M} \).
\end{remark}

\section*{Acknowledgments}
L. Baldelli is member of the {\em Gruppo Nazionale per l'Analisi Ma\-te\-ma\-ti\-ca, la Probabilit\`a e le loro Applicazioni}
(GNAMPA) of the {\em Istituto Nazionale di Alta Matematica} (INdAM).
L. Baldelli was partly supported by INdAM-GNAMPA Project 2023 titled {\em Problemi ellittici e parabolici con termini di reazione singolari e convettivi} (E53C22001930001), by National Science Centre, Poland (Grant No. 2020/37/B/ST1/02742) and by the ``Maria de Maeztu'' Excellence Unit IMAG, reference CEX2020-001105-M, funded by MCIN/AEI/10.13039/501100011033/. J. Mederski was partly supported by National Science Centre, Poland (Grant No. 2020/37/B/ST1/02742). B. Bieganowski was partly supported by National Science Centre, Poland (Grant No. 2022/47/D/ST1/00487).

\end{document}